\documentclass[11pt]{article}%
\usepackage{amssymb, amsmath, latexsym, mathrsfs, verbatim, calc}
\usepackage[latin1]{inputenc}
\usepackage{amsmath}
\usepackage{amsfonts}
\usepackage{amssymb}
\usepackage{graphicx}
\usepackage{ulem}
\usepackage{color}%
\providecommand{\U}[1]{\protect\rule{.1in}{.1in}}

\newtheorem{theorem}{Theorem}

\newtheorem{definition}[theorem]{Definition}

\newtheorem{lemma}[theorem]{Lemma}

\newtheorem{proposition}[theorem]{Proposition}
\newtheorem{remark}[theorem]{Remark}

\newenvironment{proof}[1][Proof]{\noindent\textbf{#1.} }{\ \rule{0.5em}{0.5em}}

\newcommand{\N}{\mathbb{N}}

\newcommand{\R}{\mathbb{R}}

\def\N{\mathbb{N}}
\def\R{\mathbb{R}}

\def\E{\mathbb{E}}
\def\U{\mathbb{U}}

\begin{document}

\title{Approximation of a degenerate semilinear PDEs with a nonlinear Neumann boundary condition}

\author{Khaled Bahlali\thanks{Universit\'e de Toulon, IMATH, EA $2134$, $83957$ La Garde
cedex, France.}
\and Brahim Boufoussi \thanks{Department of Mathematics, Faculty of Sciences Semlalia, Cadi Ayyad University, 2390 Marrakesh, Morocco}
\and Soufiane Mouchtabih \footnotemark[1] \footnotemark[2]}
\date{}
\maketitle

\date{}
\maketitle


\noindent\textbf{Abstract}:
We consider a system of semilinear partial differential equations (PDEs) with a nonlinearity depending  on both the solution and its gradient. The Neumann boundary condition depends on the solution in a nonlinear manner. The uniform ellipticity is not required to the diffusion coefficient. We show that this problem admits a viscosity solution which can be approximated by a penalization. The Lipschitz condition is required to the coefficients of the diffusion part. The nonlinear part as well as the Neumann condition are Lipschitz. Moreover, the nonlinear part is assumed monotone in the solution variable.
Note that the existence of a viscosity solution to this problem has been established in \cite{PardZhang} then completed in \cite{PardRas}. In the present paper, We construct a sequence of penalized system of decoupled forward backward stochastic differential equations (FBSDEs) then we directly show its strong convergence. This allows us to deal with the case where the nonlinearity depends on both the solution and its gradient.  Our work extends, in particular, the result of \cite{BouCas} and, in some sense, those of \cite{BahMatZal, bbm}. In contrast to works \cite{BahMatZal,  bbm, BouCas}, we do not pass by the weak compactness of the laws of the stochastic system associated to our problem.


\section{Introduction}
Let $D$ be a regular convex, open and bounded subset of $\mathbb{R}^d$. We introduce the function
$ \rho\in{\mathcal C}^{1}_{b}(\R^{d}) $ such that $ \rho=0 $ in $ \bar D $, $ \rho>0 $ in $ \R^{d}\setminus
{\bar D}$ and  $ \rho(x)=(d(x, {\bar D}))^{2} $ in
a neighborhood of ${\bar D}$. On the other hand, since
the domain $D$ is smooth (say ${\mathcal C^{3}}$), it is
possible to consider an extension $ l\in{\mathcal C}^{2}_{b}(\R^{d}) $
of the function $ d(\cdot\, ,\,\partial D) $ such that
$D$ and $\partial D$ are characterized by
\begin{equation}
D=\{x\in\R^{d}\,:\, l(x)>0\}\, \quad \text{and}\quad
\partial D=\{x\in\R^{d}\,:\, l(x)=0\}\,,\nonumber
\end{equation}
and for every $x\in\partial D$, $\nabla l(x)$ coincides with the
unit normal pointing toward the interior of $D$ (see for example \cite[Remark 3.1]{LionsSzn}). In particular we may and do choose $\rho$ and
$l$ such that
\begin{equation} \left<\nabla l(x)\,,\,\delta(x) \right> \leq 0\,,
\text{\quad for all $x\in\R^{d}$}\,, \label{E:penalization1} \end{equation}
where $\delta(x):=\nabla\rho(x)$ and is called the penalization term.
We have
$$\frac{1}{2}\delta(x)=x-\pi_{\bar{D}}(x),\quad \forall x\in\mathbb{R}^d $$
where $\pi_{\bar{D}}$ is the projection operator on $\bar D$. Consider the second-order differential operator
\begin{equation*}
\mathcal{L}=\frac{1}{2}\sum_{i,j}\left( \sigma\sigma^*(.)\right)_{ij}\frac{\partial^2}{\partial x_i \partial x_j}+\sum_{i}b_i(.)\frac{\partial}{\partial x_i}
\end{equation*}
where $b:\mathbb{R}^d\to \mathbb{R}^d$ and $\sigma:\mathbb{R}^d\to \mathbb{R}^{d\times d'}$ are given functions satisfying suitable assumptions.

 Our first aim is to establish the existence of a viscosity solution via penalization  to the following system of partial differential equations with nonlinear Neumann boundary condition, defined for $1\leq i\leq m$, $0\leq t\leq T$, $x\in D$.
 \begin{equation}\label{Nuemann-PDE-introduction}
 	\left\{
 	\begin{matrix}
 		\displaystyle\frac{\partial u_{i}}{\partial t}(t,x) +\mathcal{L}u_{i}(t,x)+
 		f_{i}(t,x,u(t,x),(\nabla u\sigma)(t,x))  = 0\,
 		\hfill\\
 		\noalign{\medskip}
 		u(T,x)  =  g(x)\,,\ \  x\in D\hfill\\
 		\displaystyle\frac{\partial u}{\partial n}(t,x)+h(t, x, u(t,x))=0
 		\,,\,\,
 		(t,x)\in[0,T)\times\partial D\,.
 	\end{matrix}\right.
 \end{equation}
To this end, we consider a sequence $(u^n)$ of viscosity solutions  of the following semi-linear partial
differential equations ($1\leq i\leq m$, $0\leq t\leq T$, $x\in\R^{d}$,
$n\in\N$).
\begin{equation}\label{edppen}
\left\{
\begin{aligned}
\displaystyle
&\frac{\partial u^{n}_{i}}{\partial t}(t,x) +\mathcal{L}\, u^{n}_{i}(t,x)+
f_{i}(t,x,u^{n}(t,x),(\nabla u^n\sigma)(t,x))\\
&\qquad\qquad -n\,<\delta(x),\nabla
u^{n}_{i}(t,x)+\nabla l(x)>
h_{i}(t, x, u^n(t,x))  = 0\,; \\
&u^{n}(T, x)  =  g(x)\, .\hfill\cr
\end{aligned}\right\}
\end{equation}
then we show that for each $n$, equation \eqref{edppen} has a viscosity solution $u_n$ which converges to a function $u$, and $u$ is a viscosity solution to \eqref{Nuemann-PDE-introduction}. Our method is probabilistic.

The authors of \cite{BahMatZal, bbm, BouCas} considered the case where $f$ does not depends on $\nabla u$.  Using the connection between backward stochastic differential equations (BSDEs) and partial differential equations (PDEs), the convergence of $u^n$ to $u$ has been established in \cite{BouCas} for bounded and uniformly Lipschitz coefficients $b$ and $\sigma$. The authors of \cite{BahMatZal} extended the result of \cite{BouCas} to the case where $b$ and $\sigma$ are bounded continuous. The case where $b$, $\sigma$ and $f$ are bounded measurable is considered in \cite{bbm} in the framework of $L^p$-viscosity solution. The techniques developed in the previous works rely on tightness properties of the associated sequence of BSDEs in the Jakubowski $S$-topology.  The main drawback of this method is that it does not allow to deal with the nonlinearity $f$ depending on $\nabla u$. Here, our method is direct and does not pass by weak compactness properties.  Usually, when the nonlinearity $f$ depends in the gradient of the solution, PDEs techniques are used to control the gradient $\nabla u$ in order to get the convergence of the associated BSDE. And generally, a uniform ellipticity of the diffusion is required to get a good control of the gradient $\nabla u$, see for instance \cite{bep2017, DThese, Darticlehomaop}  where this method is used in homogenization of nonlinear PDEs. Our approach is completely different:  We use a purely probabilistic method, which allows us to deal with (possibly) degenerate PDEs. The convergence of the penalized BSDE is provided only by the convergence of the penalized forward SDE. Our proof essentially use \cite[Proposition 6.80, Annex C]{PardRas2014}. The latter has been already used by the authors of \cite{PardRas} in order to prove the continuity of the solution of a system of SDE-BSDE in its initial data $(t,x)$. By bringing  essential  modifications in the  idea developed in \cite{PardRas}, we prove the convergence of our sequence of penalized BSDEs.

The paper is organized as follows: Section 2 contains some facts about reflected stochastic differential equations (SDEs) and generalized BSDEs. This mainly consists in  approximation, existence, uniqueness results and a priori estimates of the solutions. Section 3 is devoted to the penalization of the nonlinear Neumann PDE.

\section{Preliminaries and formulation of the problem}
Throughout the paper, for a fixed $T>0$, $(W_t; t\in [0,T])$ is a $d -$dimensional Brownian motion defined on a complete probability space $(\Omega,\mathcal{F},\mathbb{P})$ and for every $t\in[0,T]$, $\mathcal{F}^t_s$ is the $\sigma-$algebra $\sigma(W_r;\, t\le r\le s)\vee\mathcal{N}$ if $s\ge t$ and $\mathcal{F}^t_s=\mathcal{N}$ if $s\le t$, where $\mathcal{N}$ is the $\mathbb{P}-$zero sets of $\mathcal{F}$. For $q\ge 0$, we denote by $\mathcal{S}^q_d[0,T]$  the space of continuous progressively measurable stochastic processes $X:\Omega\times [0,T]\to \mathbb{R}^d$, such that for $q>0$ we have
$$\mathbb{E}\sup_{t\in[0,T]}|X_t|^q<+\infty.$$
For $q\ge 0$, we denote by $\mathcal{M}^q_d(0,T)$ the space of progressively measurable stochastic processes $X:\Omega\times [0,T]\to \mathbb{R}^d$ such that:
$$\mathbb{E}\left[\left(\int_{0}^{T}|X_t|^2dt \right)^{\frac{q}{2}}  \right]<+\infty\quad \mbox{if}\quad q>0; \quad and \quad  \int_{0}^{T}|X_t|^2dt<+\infty\, \quad\mathbb{P}-a.s.\quad \mbox{if}\quad q=0. $$
\subsection{Penalization for reflected stochastic differential equation}
 Let $(t,x)\in [0,T]\times \bar{D}$. The reflected SDE  under consideration is
\begin{equation}\label{rsde2}
\left\{
\begin{aligned}
& X^{t,x}_{s}=x+\int_{t}^{s} b(X^{t,x}_{r})\, dr+\int_{t}^{s}\sigma(X^{t,x}_{r})\, dW_{r}+K^{t,x}_{s}, \\
& K^{t,x}_s=\int_t^s\nabla l(X^{t,x}_r)d|K^{t,x}|_{[t,r]},\\
&|K^{t,x}|_{[t,s]}=\int_t^s1_{\{X^{t,x}_r\in \partial D\}}d|K^{t,x}|_{[t,r]},\quad s\in [t,T],
\end{aligned}
\right.
\end{equation}
 where the notation $|K^{t,x}|_{[t,s]}$ stands for the total variation of $K^{t,x}$ on the interval $[t,s]$, we denote this continuous increasing process by $k^{t,x}_s$. In particular we have
 \begin{equation}\label{k}
 k^{t,x}_s=\int_t^s<\nabla l(X^{t,x}_r),dK^{t,x}_r>.
 \end{equation}
Several authors have studied the problem of the existence of solutions of the reflected diffusion and its approximation by solutions of equations with penalization terms, we refer for example to \cite{LionsSzn,LauSlo,slominski2013,StrVara,Tanaka}. We consider the following sequence of penalized SDEs associated with our reflected diffusion $X^{t,x}$
\begin{equation}\label{PenaSDE}
X^{t,x,n}_s=x+\int_t^s\left[b(X^{t,x,n}_r)-n\delta(X^{t,x,n}_r) \right]dr+\int_t^s\sigma(X^{t,x,n}_r)dW_r,\quad s\in [t,T].
\end{equation}
For  $s\in [t,T]$, we put
\begin{eqnarray}\label{K^n, k^n}
K^{t,x,n}_s:=\int_t^s-n\delta(X^{t,x,n}_r)dr\quad\mbox{and}\quad
k^{t,x,n}_s:=\int_t^s<\nabla l(X^{t,x,n}_r),dK^{t,x,n}_r>.
\end{eqnarray}
We introduce the following assumption
\begin{itemize}
\item[$(A.1):$] There exist positive constants $C$ and $\mu$ such that for every $(x,y)\in\mathbb{R}^d$:
\begin{itemize}
	\item[(i)] $|b(x)| + |\sigma(x)\|\le C(1+|x|),$
	\item[(ii)]$|b(x)-b(y)| + |\sigma(x)-\sigma(y)\|\le \mu|x-y|.$
\end{itemize}
\end{itemize}
 It is known that under assumption $(A.1)$ equation (\ref{PenaSDE}) admits, for any fixed $n\in \mathbb{N}$, a unique strong solution, and we have for every $q\geq 1$:
\begin{equation}\label{boundXn}
\sup_{n\ge 0}\,\mathbb{E}\sup_{s\in[t,T]}|X^{t,x,n}_s|^{2q}+ \sup_{n\ge 0}\,\mathbb{E}\sup_{s\in[t,T]}|K^{t,x,n}_s|^{2q}+\sup_{n\ge 0}\mathbb{E}|K^{t,x,n}|^q_{[t,T]}<+\infty.
\end{equation}
The proof of the previous estimates can be found e.g. in   \cite[Lemma 3.1]{BahMatZal}.

The first assertion of the following theorem  is proved in \cite{Tanaka}, while the second one follows from \cite{slominski2013}.
\begin{theorem} \label{Tanaka-Slominski}
	Under assumption \rm{(A.1)}, we have
	
	(i) \ the system  (\ref{rsde2}) admits a unique solution,
	
	 (ii)  \ for every
	$ 1\leq q<\infty $ and $ 0<T<\infty $,

	$$\displaystyle \E\left[\sup_{t\leq s\leq T}| X^{t,x,n}_{s}-X^{t,x}_{s}|^{q}\right] \longrightarrow 0,\,\,\, \ as \ n\to\infty, $$ 	
the limit is uniform in $(t,x)\in[0,T]\times\bar D$.
\end{theorem}

We extend  the processes $(X^{t,x},K^{t,x})$ and $(X^{t,x,n},K^{t,x,n})$ to $[0,t)$ by putting
\begin{equation*}
X^{t,x}_s=X^{t,x,n}_s:=x,\quad K^{t,x}_s=K^{t,x,n}_s:=0,\quad \hbox{for} \ s\, \in\,[0,t).
\end{equation*}
As a consequence of Theorem \ref{Tanaka-Slominski}, we have the following convergence, which is established in \cite[Lemma 2.2]{BouCas}.
\begin{lemma}\label{lemma3} {(\cite{BouCas})} Under assumptions of Theorem \ref{Tanaka-Slominski}, we have, for any $q\ge 1$:
	\begin{align}
	{\rm (i)}&\quad
	\lim_{n\to\infty}\E\left[\sup_{0\leq s\leq T}\left|K^{t,x,n}_{s}-K^{t,x}_{s}\right|^q
	\right]=0;\,
	\nonumber\\
	{\rm (ii)}&\quad
	\lim_{n\to\infty}\E\left[\sup_{0\leq s\leq T}\left|\int_{0}^{s}
	\varphi(X^{n}_{r})\, dK^{t,x,n}_{r}-
	\int_{0}^{s}\varphi(X_{r})\, dK^{t,x}_{r}\right|^{q}\right]=0\, \quad
	\hbox{for every} \ \ \varphi\in{\mathcal C}^{1}_{b}(\R^{d}).\nonumber
	\end{align}	
\end{lemma}
\begin{remark}\label{rq cv}
	\begin{itemize}
		\item[\rm (i)] Using Lemma \ref{lemma3} and the representations (\ref{k}), (\ref{K^n, k^n}), it holds that for any $q\ge 1$ and any $(t,x)\in[0,T]\times\bar D$
		\begin{equation}
		\lim_{n\to+\infty}\mathbb{E}\sup_{s\in[0,T]}|k^{t,x,n}_s-k^{t,x}_s|^q=0.
		\end{equation}
		\item[\rm (ii)] From \cite[Corollary 2.5]{slominski}, it follows that for each $q\ge 1$ and each $(t,x)\in[0,T]\times\bar D$
		\begin{eqnarray}\label{bound X and k}
		\mathbb{E}\sup_{s\in [0,T]}|X^{t,x}_s|^{2q}+\mathbb{E}\sup_{s\in [0,T]}|K^{t,x}_s|^{2q}+ \mathbb{E}|K^{t,x}|^q_{[0,T]}<+\infty.
		\end{eqnarray}
	\end{itemize}
\end{remark}

 \subsection{Backward inequality}
We state a lemma which is a version of \cite[Proposition 6.80, Annex C]{PardRas2014}. This lemma is essential in our proofs. We give its proof for the convenience.
\begin{lemma}\label{Lemma crucial}
	Let $(Y,Z)\in \mathcal{S}^0_m\times \mathcal{M}^0_{m\times d}$, satisfying
	\begin{equation*}
	Y_t=Y_T+\int_{t}^{T}d\mathcal{K}_r-\int_{t}^{T}Z_rdW_r, \quad 0\le t\le T, \ \mathbb{P}-a.s.,
	\end{equation*}
	where $\mathcal{K}\in \mathcal{S}^0_m $ and $\mathcal{K}_.(\omega)\in BV([0,T];\mathbb{R}^m)$ (the space of bounded variation processes).
	
	Assume be given
	\begin{itemize}
		\item  a non-decreasing stochastic process $L$ with $L_0=0$,
		\item a stochastic process $R_.(\omega)\in BV([0,T],\mathbb{R}),\ $ with  $R_0=0$,
		\item  a continuous stochastic process $ V_.(\omega)\in BV([0,T],\mathbb{R}),\ $ such that $V_0=0$,
		$$\mathbb{E}\left(\sup_{s\in [0,T]}\int_{s}^{T}e^{2V_r}dR_r \right)<\infty,$$
	\end{itemize}
	and
	\begin{itemize}
		\item[$(i)$] $<Y_r,d\mathcal{K}_r>\le \frac{\alpha}{2}\|Z_r\|^2dr+|Y_r|^2dV_r+ |Y_r|dL_r+dR_r$ as measures on $[0,T]$,
		\item[$(ii)$]$\mathbb{E}\sup_{r\in [0,T]}e^{2V_r}|Y_r|^2<+\infty$.
	\end{itemize}
	We  have the following conclusion : if $\alpha<1$, then there exist positive constants $C_1$, $C_2$ and $C_3$, depending only on $\alpha$, such that
	\begin{eqnarray*}
		&&	\mathbb{E}\left(\sup_{r\in[0,T]}|e^{V_r}Y_r|^2 \right)+\mathbb{E}\left(\int_{0}^{T}e^{2V_r}\|Z_r\|^2dr\right)\nonumber\\
		&& \qquad \le C_1\; \mathbb{E}|e^{V_T}Y_T|^2+C_2\mathbb{E}\left( \int_{0}^{T}e^{V_r}dL_r\right)^2+C_3\mathbb{E}\sup_{s\in[0,T]}\int_{s}^{T}e^{2V_r}dR_r.
	\end{eqnarray*}		
\end{lemma}

\begin{proof}
	By It\^o's formula, we have
	\begin{eqnarray*}
		 &&|e^{V_t}Y_t|^2=|e^{V_T}Y_T|^2-2\int_{t}^{T}e^{2V_s}|Y_s|^2dV_s+2\int_{t}^{T}e^{2V_s}<Y_s,d\mathcal{K}_s>\\
		&&-\int_{t}^{T}e^{2V_s}\|Z_s\|^2ds-2\int_{t}^{T}<e^{V_s}Y_s,e^{V_s}Z_sdW_s>\\
		&&=|e^{V_T}Y_T|^2+2\int_{t}^{T}e^{2V_s}\left( <Y_s,d\mathcal{K}_s>-|Y_s|^2dV_s\right)-\int_{t}^{T}e^{2V_s}\|Z_s\|^2ds\\
		&&-2\int_{t}^{T}<e^{V_s}Y_s,e^{V_s}Z_sdW_s>.
	\end{eqnarray*}
	Using (i) of Lemma \ref{Lemma crucial}, we get
	\begin{eqnarray*}
		&&|e^{V_t}Y_t|^2+(1-\alpha)\int_{t}^{T}e^{2V_s}\|Z_s\|^2ds\le\\ &&\qquad\qquad\qquad|e^{V_T}Y_T|^2+2\int_{t}^{T}e^{2V_s}\left(dR_s+|Y_s|dL_s\right)-2\int_{t}^{T}<e^{V_s}Y_s,e^{V_s}Z_sdW_s>.
	\end{eqnarray*}
	We consider the following sequence of stopping time
	\begin{eqnarray*}
		T_n:=T\wedge\inf\{s\ge t: \sup_{r\in [t,s]}|e^{V_r}Y_r-e^{V_t}Y_t|+\int_{t}^{s}e^{2V_r}\|Z_r\|^2dr+\int_{t}^{s}e^{V_r}dL_r\ge n\}.
	\end{eqnarray*}
	For $s\in[t, \ T]$, we put
	$ N_s:=2\int_{0}^{s}1_{[t,T_n]}(r)<e^{V_r}Y_r,e^{V_r}Z_rdW_r>$.
	\\
	Since
	\begin{eqnarray*}
		&&\mathbb{E}\left( <N>_T\right)^{\frac{1}{2}}\le C\mathbb{E}\left(\int_{t}^{T_n}e^{4V_r}|Y_r|^2\|Z_r\|^2dr \right)^{\frac{1}{2}}\\
		&&\le \sqrt{2}C\mathbb{E}\left(\left[ \sup_{r\in [t,T_n]}|e^{V_r}Y_r-e^{V_t}Y_t|^2+|e^{V_t}Y_t|^2\right] \int_{t}^{T_n}e^{2V_r}|Z_r|^2dr \right)^{\frac{1}{2}}\\
		&& \le \sqrt{2}C\mathbb{E}\left(|e^{V_t}Y_t|+n \right)\sqrt{n}<+\infty.
	\end{eqnarray*}
Then, the process $\{N_s:s\in [0,T]\}$ is a martingale.
\\
	By the forgoing, we have
	\begin{eqnarray*}
		&&|e^{V_t}Y_t|^2+(1-\alpha)\int_{t}^{T_n}e^{2V_s}\|Z_s\|^2ds\le |e^{V_{T_n}}Y_{T_n}|^2+2\int_{t}^{T_n}e^{2V_s}\left(dR_s+|Y_s|dL_s\right)-(N_{T_n}-N_t).
	\end{eqnarray*}
	Taking expectation, it follows that
	\begin{eqnarray}\label{estimate Int Z}
	&&\mathbb{E}\left( |e^{V_t}Y_t|^2+(1-\alpha)\int_{t}^{T_n}e^{2V_s}\|Z_s\|^2ds\right) \le\nonumber\\
	&&\qquad\qquad\qquad \mathbb{E}\left(  |e^{V_{T_n}}Y_{T_n}|^2+2\int_{t}^{T_n}e^{2V_s}\left(dR_s+|Y_s|dL_s\right)\right).
	\end{eqnarray}
	On the other hand, we have
	\begin{eqnarray*}
		&&\mathbb{E}\sup_{r\in[0,T_n]}|e^{V_r}Y_r|^2\le \mathbb{E}\left( |e^{V_{T_n}}Y_{T_n}|^2
		+2\sup_{s\in[0,T_n]}\int_{s}^{T_n}e^{2V_r}dR_r+ 2\int_{0}^{T_n}e^{2V_s}|Y_s|dL_s+2\sup_{r\in [0,T_n]}|N_r|\right).
	\end{eqnarray*}
	The Burkholder-Davis-Gundy inequality shows that
	\begin{align*}
		\mathbb{E}\sup_{r\in[0,T_n]}|e^{V_r}Y_r|^2
		& \le \mathbb{E}\left( |e^{V_{T_n}}Y_{T_n}|^2+2\sup_{s\in[0,T_n]}\int_{s}^{T_n}e^{2V_r}dR_r+ 2\int_{0}^{T_n}e^{2V_s}|Y_s|dL_s\right)
		\notag
		\\
		& \ \quad +2C_{BDG}\mathbb{E}\left( \int_{0}^{T_n}e^{4V_r}|Y_r|^2\|Z_r\|^2dr\right)^{\frac{1}{2}} \notag
		\\
		& \le \mathbb{E}\left( |e^{V_{T_n}}Y_{T_n}|^2+2\sup_{s\in[0,T_n]}\int_{s}^{T_n}e^{2V_r}dR_r+ 2\int_{0}^{T_n}e^{2V_s}|Y_s|dL_s\right)
		\nonumber
		\\
		& \ \quad  +2C_{BDG}\mathbb{E}\left(\sup_{r\in [0,T_n]}e^{2V_r}|Y_r|^2 \int_{0}^{T_n}e^{2V_r}\|Z_r\|^2dr\right)^{\frac{1}{2}}
		\nonumber
	\\
		& \le
		\mathbb{E}\left( |e^{V_{T_n}}Y_{T_n}|^2+2\sup_{s\in[0,T_n]}\int_{s}^{T_n}e^{2V_r}dR_r+ 2\int_{0}^{T_n}e^{2V_s}|Y_s|dL_s\right)
		\nonumber
		\\
		& \ \quad +\mathbb{E}\left( \frac{1}{2}\sup_{r\in [0,T_n]}e^{2V_r}|Y_r|^2+2C_{BDG}^2 \int_{0}^{T_n}e^{2V_r}\|Z_r\|^2 dr\right) \nonumber
		\\
		& \le 2\mathbb{E}\left( |e^{V_{T_n}}Y_{T_n}|^2+2\sup_{s\in[0,T_n]}\int_{s}^{T_n}e^{2V_r}dR_r + 2\int_{0}^{T_n}e^{2V_s}|Y_s|dL_s\right)
		\nonumber
		\\
		& \ \quad  + 4C_{BDG}^2\mathbb{E}\left(  \int_{0}^{T_n}e^{2V_r}\|Z_r\|^2 dr\right),
	\end{align*}
	where $C_{BDG}$ denotes the Burkholder-Davis-Gundy constant.
\\
From inequality (\ref{estimate Int Z}), we get
	\begin{eqnarray*}
		 &&\mathbb{E}\sup_{r\in[0,T_n]}|e^{V_r}Y_r|^2+\frac{1}{2}\mathbb{E}\int_{0}^{T_n}e^{2V_r}\|Z_r\|^2dr\le \mu\,\mathbb{E}\left( |e^{V_{T_n}}Y_{T_n}|^2\right)\\
		&& +2\mu\,\mathbb{E}\left( \sup_{s\in[0,T_n]}\int_{s}^{T_n}e^{2V_r}dR_r\right) + 2\mu\,\mathbb{E}\left( \int_{0}^{T_n}e^{2V_s}|Y_s|dL_s\right),
	\end{eqnarray*}	
	where
	$$\mu =2+\frac{2C_{BDG}^2}{1-\alpha}+\frac{1}{2(1-\alpha)}.$$
	It follows that
	\begin{align*}
		 &\mathbb{E}\sup_{r\in[0,T_n]}|e^{V_r}Y_r|^2+\frac{1}{2}\mathbb{E}\int_{0}^{T_n}e^{2V_r}\|Z_r\|^2dr\le \mathbb{E}\left( \mu\,|e^{V_{T_n}}Y_{T_n}|^2+2\mu\sup_{s\in[0,T_n]}\int_{s}^{T_n}e^{2V_r}dR_r\right)  \\
		&\qquad\qquad+2\mu\mathbb{E}\left(\sup_{r\in [0,T_n]}e^{V_r}|Y_r|\int_{0}^{T_n}e^{V_r}dL_r\right) \\
		 &\mathbb{E}\sup_{r\in[0,T_n]}|e^{V_r}Y_r|^2+\frac{1}{2}\mathbb{E}\int_{0}^{T_n}e^{2V_r}\|Z_r\|^2dr\le \mathbb{E}\left( \mu|e^{V_{T_n}}Y_{T_n}|^2+2\mu\sup_{s\in[0,T_n]}\int_{s}^{T_n}e^{2V_r}dR_r\right) \\
		&\qquad\qquad\qquad\qquad\qquad +\frac{1}{2}\mathbb{E}\sup_{r\in [0,T_n]}e^{2V_r}|Y_r|^2+2\mu^2\mathbb{E}\left( \int_{0}^{T_n}e^{V_r}dL_r\right)^2 \\
		&\mathbb{E}\sup_{r\in[0,T_n]}|e^{V_r}Y_r|^2+\mathbb{E}\int_{0}^{T_n}e^{2V_r}\|Z_r\|^2dr\le \mathbb{E}\left( 2\mu|e^{V_{T_n}}Y_{T_n}|^2
		+4\mu\sup_{s\in[0,T_n]}\int_{s}^{T_n}e^{2V_r}dR_r\right)\\
		&\qquad\qquad\qquad\qquad  +4\mu^2\mathbb{E}\left( \int_{0}^{T_n}e^{V_r}dL_r\right)^2.
	\end{align*}
	Put
	$$C_1:=2\mu,\qquad C_2:=4\mu \qquad\mbox{and}\qquad C_3:=4\mu^2.$$
	We then have,
	\begin{align*}
		\mathbb{E}\left( \sup_{r\in[0,T_n]}|e^{V_r}Y_r|^2+\int_{0}^{T_n}e^{2V_r}\|Z_r\|^2dr\right)
&\le C_1\mathbb{E}|e^{V_{T_n}}Y_{T_n}|^2
		+C_2\mathbb{E}\sup_{s\in[0,T_n]}\int_{s}^{T_n}e^{2V_r}dR_r
\\
& \ \ \ \ + C_3\mathbb{E}\left( \int_{0}^{T}e^{V_r}dL_r\right)^2.
	\end{align*}
Letting $n$ tends to $ +\infty$, we conclude by using the Beppo-Levi theorem for the left-hand side term of the previous inequality, and the Lebesgue dominated convergence theorem for the right-hand side term.
\end{proof}

\subsection{The BSDEs associated to the nonlinear Neumann problem}

We introduce the generalized BSDEs which we have to use.
Let $f:\, [0,T]\times\R^{d}\times\R^{m}\times \mathbb{R}^{m\times d'}
\rightarrow\R^{m} $, $h:\, [0,T]\times\R^{d}\times\R^{m}\rightarrow\R^{m} $ and  $ g\,:\,\R^{d}\rightarrow\R^{m} $ be continuous functions, satisfying the following assumptions:
\begin{itemize}
	\item[$(A.4)$] There exist $C$, $l_f$ positive constants and $\beta <0$, $\mu_f\in \mathbb{R}$  such that for every $t\in[0,T]$ and every $(x, x^{\prime}, y, y^{\prime},z,z')
	\in\left(\mathbb{R}^{d}\right)^{2}\times\left(\mathbb{R}^{m}\right)^{2}\times (\mathbb{R}^{m\times d'})^2 $ we have:
	\begin{itemize}
		\item[(i)]	$<y-y',f(t, x, y,z)-f(t, x, y',z)>\,\leq \mu_f\, |y-y'|^2,$
		\item[(ii)] $|f(t, x, y,z)-f(t, x, y,z')|\le l_f \|z-z'\|,$
		\item[(iii)] $ |f(t,x,y,0)|\leq C\, (1+|y|), $	
		\item[(iv)] 	$<y-y',h(t,x,y)-h(t,x,y')>\le \beta|y-y'|^2,$
		\item[(v)] $ |h(t,x,y)|\leq C\, (1+|y|), $
		\item[(vi)] $|g(x)|\le C(1+|x|).$
	\end{itemize}	
\end{itemize}
For every $t\in [0,T]$ and $s\in [t,T]$, consider the following  generalized BSDEs
\begin{eqnarray}
&& Y^{t,x,n}_{s}=g(X^{t,x,n}_{T})+
\int_{s}^{T} f(r, X^{t,x,n}_{r}, Y^{t,x,n}_{r},Z^{t,x,n}_r)\, dr
+\int_s^Th(r,X^{t,x,n}_r,Y^{t,x,n}_r)dk^{t,x,n}_r \nonumber\\
&&\qquad \qquad\qquad -\int_{s}^{T}Z^{t,x,n}_{r}\, dW_{r}\label{bakseq},
\end{eqnarray}
and
\begin{equation}
Y^{t,x}_{s}=g(X^{t,x}_{T})+
\int_{s}^{T} f(r, X^{t,x}_{r}, Y^{t,x}_{r},Z^{t,x}_r)\, dr+\int_s^Th(r,X^{t,x}_r,Y^{t,x}_r)dk^{t,x}_r
-\int_{s}^{T}Z^{t,x}_{r}\, dW_{r}.
\label{baklimit}
\end{equation}
 According to \cite{Pard,PardZhang}, assumption $(A.4)$ ensures the existence of unique solutions to equations $(\ref{bakseq})$ and $(\ref{baklimit})$. The solutions of equations $(\ref{bakseq})$ and $(\ref{baklimit})$ will be respectively denoted by  $(Y_s^{t,x,n},Z^{t,x,n}_s)_{s\in [t,T]}$ and $(Y_s^{t,x},Z^{t,x}_s)_{s\in [t,T]}$.
\begin{lemma}
	Under assumptions \rm{(A.1)(i)} and \rm{(A.4)}, it holds that for any $t\in [0,T]$,
	\begin{equation}\label{bound Y^n }
	\sup_{n\ge 1}\mathbb{E}\left(\sup_{r\in [t,T]}|Y^{t,x,n}_r|^2+\int_{t}^{T}|Y^{t,x,n}_r|^2dk^{t,x,n}_r+\int_{t}^{T}\|Z^{t,x,n}_r\|^2dr \right) <+\infty
	\end{equation}
and
\begin{equation}\label{bound Y}
\mathbb{E}(\sup_{r\in [t,T]}|Y^{t,x}_r|)^q+\mathbb{E}\left( \int_{t}^{T}\|Z^{t,x}_r\|^2dr\right)^{\frac{q}{2}}  <+\infty, \qquad \forall\, q> 1.
\end{equation} 	
\end{lemma}
\begin{proof}
We prove the first assertion. It\^o's formula gives
\begin{eqnarray*}
&&|Y^{t,x,n}_s|^2+\int_{s}^{T}\|Z^{t,x,n}_r\|^2dr=|g(X^{t,x,n}_T)|^2\\
&&\qquad+2\int_{t}^{T}<Y^{t,x,n}_r,f(r,X^{t,x,n}_r,Y^{t,x,n}_r,Z^{t,x,n}_r)>dr\\
&&\quad\quad+2\int_{t}^{T}<Y^{t,x,n}_r,h(r,X^{t,x,n}_r,Y^{t,x,n}_r)>dk^{t,x,n}_r-2\int_{t}^{T}<Y^{t,x,n}_r,Z^{t,x,n}_rdW_r>.
\end{eqnarray*}
Using assumption (A.4), we find
\begin{eqnarray*}
	&&|Y^{t,x,n}_s|^2+\int_{s}^{T}\|Z^{t,x,n}_r\|^2dr\le C^2T+2C^2+2C^2|X^{t,x,n}_T|^2\\
	&&\qquad\qquad\quad+(1+2\mu_f+2l_f^2)\int_{s}^{T}|Y^{t,x,n}_r|^2dr+2\int_{s}^{T}\left( \beta|Y^{t,x,n}_r|^2+C|Y^{t,x,n}_r|\right) dk^{t,x,n}_r\\
	&&\qquad\qquad\qquad\qquad\qquad\qquad-2\int_{s}^{T}<Y^{t,x,n}_r,Z^{t,x,n}_r dW_r>.
\end{eqnarray*}
Since  $-\beta > 0$, we use the inequality $2ab\le (-\beta) a^2+\frac{b^2}{(-\beta)}$,\; to get
\begin{eqnarray*}
	&&|Y^{t,x,n}_s|^2+\int_{s}^{T}\|Z^{t,x,n}_r\|^2dr\le C^2T+C^2+C^2|X^{t,x,n}_T|^2\\
	&&\qquad\qquad\qquad +(1+2\mu_f+2l_f^2)\int_{s}^{T}|Y^{t,x,n}_r|^2dr -\frac{C^2}{\beta}k^{t,x,n}_T+\beta \int_{s}^{T}|Y^{t,x,n}_r|^2dk^{t,x,n}_r\\
	&&\qquad\qquad\qquad\qquad\qquad\qquad -2\int_{s}^{T}<Y^{t,x,n}_r,Z^{t,x,n}_rdW_r>.
\end{eqnarray*}
Since the process $(\int_{0}^{.}<Y^{t,x,n}_r,Z^{t,x,n}_r)dW_r>)$ is a uniformly integrable martingale, we take expectation in the previous inequality to show that
\begin{eqnarray*}
	&&\mathbb{E}\left( |Y^{t,x,n}_s|^2+|\beta| \int_{s}^{T}|Y^{t,x,n}_r|^2dk^{t,x,n}_r+\int_{s}^{T}\|Z^{t,x,n}_r\|^2dr\right) \le C^2T+C^2+C^2\mathbb{E}|X^{t,x,n}_T|^2\\
	&&\qquad \qquad\qquad +\frac{C^2}{|\beta|}\mathbb{E}k^{t,x,n}_T
	+(1+2\mu_f+2l_f^2)\mathbb{E}\int_{s}^{T}|Y^{t,x,n}_r|^2dr.
\end{eqnarray*}
Using estimate (\ref{boundXn}) and Gronwall's inequality, we obtain
\begin{equation*}
\sup_{n\ge 1}\sup_{s\in[t,T]}\mathbb{E}\left( |Y^{t,x,n}_s|^2+ \int_{s}^{T}|Y^{t,x,n}_r|^2dk^{t,x,n}_r+\int_{s}^{T}\|Z^{t,x,n}_r\|^2dr\right)<+\infty.
\end{equation*}
Burkholder-Davis-Gundy inequality shows that
\begin{equation*}
\sup_{n\ge 1}\mathbb{E}\left(\sup_{s\in[t,T]} |Y^{t,x,n}_s|^2+ \int_{t}^{T}|Y^{t,x,n}_r|^2dk^{t,x,n}_r+\int_{t}^{T}\|Z^{t,x,n}_r\|^2dr\right)<+\infty.
\end{equation*}
 Inequality \eqref{bound Y^n } is proved.
 Using \cite[Proposition A.2]{MatRas}, we prove    inequality  (\ref{bound Y}).
\end{proof}

\vskip 0.2cm
We extend the processes $(Y^{t,x,n},Z^{t,x,n})$ and $(Y^{t,x},Z^{t,x})$ to $[0,t)$ as follows
\begin{equation}\label{YZsoust}
	Y^{t,x,n}_s:=Y^{t,x,n}_t,\quad Y^{t,x}_s:=Y^{t,x}_t\quad\text{and}\quad Z^{t,x,n}_s=Z^{t,x}_s:=0, \quad  s\, \in\, [0,t).
\end{equation}
\section{Penalization of the nonlinear Neumann PDE}
We divide this section into two parts. The first one concerns the convergence of the solution of the BSDE (\ref{bakseq}). The second one is an application of our convergence to the nonlinear Neumann boundary problem.
\subsection{Convergence of the penalized BSDE}
For $(t,x)\in [0,T]\times\bar{D}$, let $(Y^{t,x,n}_s,Z^{t,x,n}_s)_{s\in [0,T]}$ and $(Y^{t,x}_s,Z^{t,x}_s)_{s\in [0,T]}$ be  respectively, the solutions of BSDEs (\ref{bakseq}) and (\ref{baklimit}). Our first main result is
\begin{theorem}\label{cvgcebsde}
	Let assumptions \rm{(A.1)} and \rm{(A.4)} hold.  Then, we have the following convergence
	\begin{equation*}
		\mathbb{E}\left(\sup_{r\in[0,T]}|Y^{t,x,n}_r-Y^{t,x}_r|^2 +\int_{0}^{T}\|Z^{t,x,n}_r-Z^{t,x}_r\|^2dr \right)\to 0,\quad \mbox{as}\,\, n\to +\infty.
	\end{equation*}
\end{theorem}
\begin{proof} We adapt the proof of \cite[Theorem 3.1]{PardRas} to our situation by bringing some modifications. From now on, we suppress the superscripts $(t,x)$, and $C$ will denote a nonnegative constant, which may vary from one line to another, but does not depend on $n$. We shall apply Lemma \ref{Lemma crucial} to the following BSDE
\begin{eqnarray*}
Y^n_s-Y_s&= & g(X^n_T)-g(X_T)+\int_{s}^{T}d\mathcal{K}^n_r-\int_{s}^{T}(Z^n_r-Z_r) dW_r
\end{eqnarray*}	
where
\begin{eqnarray*}
d\mathcal{K}^n_r&:=&\left[f(r,X^n_r,Y^n_r,Z^n_r)-f(r,X_r,Y_r,Z_r)\right]dr+h(r,X^n_r,Y^n_r)dk^n_r-h(r,X_r,Y_r)dk_r.
\end{eqnarray*}
Using (A.4)(i)-(A.4)(ii), we get for every $0\leq t\leq s\leq T$,
\begin{eqnarray*}
&& \int_{s}^{T}<Y^n_r-Y_r,f(r,X^n_r,Y^n_r,Z^n_r)-f(r,X_r,Y_r,Z_r)>dr\\
&&= \int_{s}^{T} <Y^n_r-Y_r,f(r,X^n_r,Y^n_r,Z^n_r)-f(r,X^n_r,Y_r,Z^n_r)>dr\\
&& \quad + \int_{s}^{T} <Y^n_r-Y_r,f(r,X^n_r,Y_r,Z^n_r)-f(r,X^n_r,Y_r,Z_r)>dr\\
 && \quad + \int_{s}^{T} <Y^n_r-Y_r,f(r,X^n_r,Y_r,Z_r)-f(r,X_r,Y_r,Z_r)>dr\\
&& \le
\int_{s}^{T} \mu_f|Y^n_r-Y_r|^2dr+ l_f|Y^n_r-Y_r|\,\|Z^n_r-Z_r\|dr\\
&& \quad + \int_{s}^{T} |Y^n_r-Y_r||f(r,X^n_r,Y_r,Z_r)-f(r,X_r,Y_r,Z_r)|dr\\
&& \le
\int_{s}^{T} (l_f^2+\mu_f)|Y^n_r-Y_r|^2dr+\frac{1}{4}\|Z^n_r-Z_r\|^2dr\\
&& \quad + \int_{s}^{T} |Y^n_r-Y_r||f(r,X^n_r,Y_r,Z_r)-f(r,X_r,Y_r,Z_r)|dr.
\end{eqnarray*}
On the other hand, thanks to assumption (A.4)(iv), we obtain
\begin{eqnarray*}
&& \int_{s}^{T} <Y^n_r-Y_r,h(r,X^n_r,Y^n_r)dk^n_r-h(r,X_r,Y_r)dk_r>\\
&&= \int_{s}^{T} <Y^n_r-Y_r, h(r,X^n_r,Y^n_r)- h(r,X^n_r,Y_r)>dk^n_r\\
&& \quad + \int_{s}^{T} <Y^n_r-Y_r,h(r,X^n_r,Y_r)- h(r,X_r,Y_r)>dk^n_r\\
&& \quad +\int_{s}^{T} <Y^n_r-Y_r,h(r,X_r,Y_r)(dk^n_r-dk_r)>\\
&& \le \beta\int_{s}^{T} |Y_r-Y^n_r|^2 dk^n_r + \int_{s}^{T} |Y^n_r-Y_r||h(r,X^n_r,Y_r)-
h(r,X_r,Y_r)|dk^n_r\\
&& \quad + \int_{s}^{T}<Y^n_r-Y_r,h(r,X_r,Y_r)>(dk^n_r-dk_r)\\
&&\le \int_{s}^{T} |Y^n_r-Y_r||h(r,X^n_r,Y_r)-
h(r,X_r,Y_r)|dk^n_r + \int_{s}^{T} <Y^n_r-Y_r,h(r,X_r,Y_r)>(dk^n_r-dk_r).
\end{eqnarray*}	

By the foregoing, it holds that, for $\lambda=(l_f^2+\mu_f)\vee l_f^2$,
\begin{eqnarray*}
<Y^n_r-Y_r, d\mathcal{K}^n_r>&\le & \frac{1}{4}\|Z^n_r-Z_r\|^2dr+\lambda|Y^n_r-Y_r|^2 dr\\
&&+|Y^n_r-Y_r||h(r,X^n_r,Y_r)-
h(r,X_r,Y_r)|dk^n_r\\
&&+|Y^n_r-Y_r||f(r,X^n_r,Y_r,Z_r)-f(r,X_r,Y_r,Z_r)|dr\\
&&+<Y^n_r-Y_r,h(r,X_r,Y_r)>(dk^n_r-dk_r)\\
&=& \frac{1}{4}\|Z^n_r-Z_r\|^2dr+\lambda|Y^n_r-Y_r|^2 dr+|Y^n_r-Y_r|dL^n_r+dR^n_r ,
\end{eqnarray*}
where $L^n$ and $R^n$ are defined by
\begin{eqnarray}
dL^n_r&:=&|f(r,X^n_r,Y_r,Z_r)-f(r,X_r,Y_r,Z_r)|dr+|h(r,X^n_r,Y_r)-
h(r,X_r,Y_r)|dk^n_r\label{L^n}\\
dR^n_r&:=& <Y^n_r-Y_r,h(r,X_r,Y_r)>(dk^n_r-dk_r).\label{R^n}
\end{eqnarray}
Therefore, by Lemma \ref{Lemma crucial}, there exist positive constants $C_1,C_2$ and $C_3$ such that
\begin{eqnarray}\label{main ineq}
&&\mathbb{E}\left(\sup_{r\in[0,T]}e^{2\lambda r}|Y^n_r-Y_r|^2 \right)+\mathbb{E}\left(\int_{0}^{T}e^{2\lambda r}\|Z^n_r-Z_r\|^2dr \right) \\
&&\le C_1\,\mathbb{E}e^{2\lambda T}|g(X^n_T)-g(X_T)|^2+C_2\mathbb{E}\left(\int_{0}^{T}e^{\lambda r}dL^n_r \right)^2+C_3\mathbb{E} \sup_{s\in [0,T]}\int_{s}^{T}e^{2\lambda r}dR^n_r. \nonumber
\end{eqnarray}
We shall give several auxiliary assertions ensuring that the right-hand side term of the previous inequality converges to zero as $n$ goes to $+\infty$.
\begin{lemma}\label{lem1}
	Under assumptions \rm{(A.1)} and \rm{(A.4)(vi)},  the following convergence holds
	$$\lim_{n\to \infty}\mathbb{E}\left( e^{2\lambda T}|g(X^n_T)-g(X_T)|^2\right)=0.$$
\end{lemma}
\begin{proof}{}
Taking into account the convergence of $X^n_T$ to $X_T$, the continuity of $g$, assumption (A.4)(vi)  and estimates (\ref{boundXn}) and  (\ref{bound X and k}), the result follows by using the Lebesgue dominated convergence theorem.
\end{proof}
\begin{lemma}\label{lem2}
	Let $L^n$ be the processes given by equation (\ref{L^n}).  Assume that {\rm{(A.1)} and \rm{(A.4)}} are satisfied. Then,
	$$\lim_{n\to \infty}\mathbb{E}\left(\int_{0}^{T}e^{\lambda r}dL^n_r \right)^2=0 .$$
\end{lemma}
\begin{proof} Using the triangular inequality, we obtain
	\begin{eqnarray*}
		\mathbb{E}\left(\int_{0}^{T}e^{\lambda r}dL^n_r \right)^2&\le& 2\mathbb{E}\left(\int_{0}^{T}e^{\lambda r}|f(r,X^n_r,Y_r,Z_r)-f(r,X_r,Y_r,Z_r)|dr \right)^2\\ &&+2\mathbb{E}\left(\int_{0}^{T}e^{\lambda r}|h(r,X^n_r,Y_r)-
		h(r,X_r,Y_r)|dk^n_r \right)^2\\
		&:=& I^n_1+I^n_2.
	\end{eqnarray*}	
We shall show that $I^n_1$ and $I^n_2$ tend to zero as $n$ tends to $\infty$. H\"older's inequality leads to
 \begin{eqnarray}\label{I^1 pena}
 I^n_1&=& \mathbb{E}\left(\int_{0}^{T}e^{\lambda r}|f(r,X^n_r,Y_r,Z_r)-f(r,X_r,Y_r,Z_r)|dr \right)^2 \nonumber\\
 &\le& 2 Te^{2\lambda T} \mathbb{E}\left(\int_{0}^{T}|f(r,X^n_r,Y_r,Z_r)-f(r,X_r,Y_r,Z_r)|^2dr \right).
 \end{eqnarray}
Again by the convergence of $X^n$ to $X$ in each $L^q$ with respect to the uniform norm and the continuity of $f$ we deduce that the sequence $|f(r,X^n_r,Y_r,Z_r)-f(r,X_r,Y_r,Z_r)|^2$ converges to zero in probability, $a.e. \;r  \in [0,T]$. Since by assumptions (A.4)(ii) and (A.4)(iii) on $f$ we have
$$|f(r,X^n_r,Y_r,Z_r)-f(r,X_r,Y_r,Z_r)|^2\le C(1+|Y_r|^2+\|Z_r\|^2),\quad a.e.\; r\in [0,T],$$
it follows that
$$\mathbb{E}|f(r,X^n_r,Y_r,Z_r)-f(r,X_r,Y_r,Z_r)|^2\to 0,\quad \mbox{as}\quad n\to +\infty,\quad a.e.\; r\in [0,T].$$
Using the Lebesgue dominated convergence theorem, we get $lim_{n\to +\infty}I^n_1=0$. Concerning $I^n_2$, H\"older's inequality yields
 \begin{eqnarray}\label{I^2 pena}
 I^n_2&=& 2\mathbb{E}\left(\int_{0}^{T}e^{\lambda r}|h(r,X^n_r,Y_r)-
 h(r,X_r,Y_r)|dk^n_r \right)^2 \nonumber\\
 &\le & 2\,e^{2\lambda T}\left(\mathbb{E}\sup_{r\in [0,T]}|h(r,X^n_r,Y_r)-
 h(r,X_r,Y_r)|^4\right)^{\frac{1}{2}} \left(\sup_{n\ge 1} \mathbb{E}\left( k^n_T\right) ^4\right)^{\frac{1}{2}}.
 \end{eqnarray}
 On the other hand, by the linear growth assumption on $h$, we have for each $q>1$
 $$\mathbb{E}\sup_{r\in [0,T]}|h(r,X^n_r,Y_r)-
 h(r,X_r,Y_r)|^{4q}\le C(1+\mathbb{E}\sup_{r\in [0,T]}|Y_r|^{4q}).$$
 It follows from estimates (\ref{bound Y}) that the sequence of random variables\\ $sup_{r\in [0,T]}|h(r,X^n_r,Y_r)-
 h(r,X_r,Y_r)|^4$ is uniformly integrable. Since
 $X^n$ converges to $X$ in each $L^q$ for the uniform norm, we deduce that the sequence $sup_{r\in [0,T]}|h(r,X^n_r,Y_r)-
 h(r,X_r,Y_r)|^4$ converges to zero in probability as $n$ goes to $+\infty$. This combined with estimate (\ref{boundXn}) ensure that $lim_{n\to +\infty}I^n_2=0$.  Lemma \ref{lem2} is proved.
\end{proof}

We will show an estimate for the solution $Y^n$ that will be used to control the term $\mathbb{E} \sup_{s\in [0,T]}\int_{s}^{T}e^{2\lambda r}dR^n_r$. To this end, let $N\in \mathbb{N},\ N>T$ and the partition of $[0,T]$, $r_i=\frac{iT}{N}$ $i=0,...,N$. We put  $r/N:=\max\{r_i;\, r_i\le r\}$. Given a continuous stochastic process $(H_r)_{r\in [0,T]}$, we define
$$H^N_r:=\sum_{i=0}^{N-1}H_{r_i}1_{[r_i,r_{i+1})}(r)+H_T1_{\{T\}}(r)=H_{r/N}.$$
\begin{lemma}\label{lem3} Assume {\rm{(A.1)} and \rm{(A.4)}} hold.
	 Then, for any $q\in ]1, \ 2[$, there exists a positive constant $C$ depending  on $T$, $q$ and independent  on $N$, such that:
	\begin{eqnarray*}
		\limsup_{n\to \infty}\mathbb{E}\left(\int_{0}^{T}|Y^n_r-Y^{n,N}_r|^q(dk^n_r+dk_r) \right)\le \frac{C}{N^{q/2}}+ C\left[\mathbb{E} \max_{i=1,\hdots, N}\left(k_{r_i}-k_{r_{i-1}} \right)^{\frac{2q}{2-q}}  \right]^{\frac{2-q}{4}}.
	\end{eqnarray*}
\end{lemma}
\begin{proof}
	We write BSDE (\ref{bakseq}) between $s/N$ and $s$
	\begin{eqnarray*}
		Y^{n,N}_s&=& Y^n_s+\int_{s/N}^{s}f(r,X^n_r,Y^n_r,Z^n_r)dr\\
		&&+\int_{s/N}^{s}h(r,X^n_r,Y^n_r)dk^n_r-\int_{s/N}^{s}Z^n_rdW_r.
	\end{eqnarray*}
	H\"older's inequality gives
	\begin{eqnarray*}
		|Y^{n,N}_s-Y^n_s|^q&\le& \frac{C}{N^{q/2}}\left[\int_{s/N}^{s}|f(r,X^n_r,Y^n_r,Z^n_r)|^2 dr \right]^{q/2}\\ &&+C\left(k^n_s-k^n_{s/N} \right)^{q/2} \left[ \int_{s/N}^{s}|h(r,X^n_r,Y^n_r)|^2dk^n_r\right]^{q/2}
		+C\left|\int_{s/N}^{s}Z^n_rdW_r \right|^q.
	\end{eqnarray*}
	It follows that
	\begin{eqnarray*}
		\mathbb{E}\left(\int_{0}^{T}|Y^n_r-Y^{n,N}_r|^q(dk^n_r+dk_r) \right)\le J^{n,N}_1+J^{n,N}_2+J^{n,N}_3
	\end{eqnarray*}
	where
	\begin{eqnarray*}
		 J^{n,N}_1:=\frac{C}{N^{q/2}}\mathbb{E}\int_{0}^{T}\left[\int_{s/N}^{s}|f(r,X^n_r,Y^n_r,Z^n_r)|^2 dr \right]^{q/2}(dk^n_s+dk_s) ,
	\end{eqnarray*}
	\begin{eqnarray*}
	J^{n,N}_2:= C\,\mathbb{E}\int_{0}^{T}(k^n_s-k^n_{s/N})^{q/2}\left(\int_{s/N}^{s}|h(r,X^n_r,Y^n_r)|^2dk^n_r \right)^{q/2}(dk^n_s+dk_s) ,
	\end{eqnarray*}
	\begin{eqnarray*}
	J^{n,N}_3:= C\,\mathbb{E}\int_{0}^{T}\left|\int_{s/N}^{s}Z^n_rdW_r\right|^q (dk^n_s+dk_s).
	\end{eqnarray*}
	We shall estimate $J^{n,N}_1$, $J^{n,N}_2$ and $J^{n,N}_3$. We use H\"older's inequality to obtain
	\begin{eqnarray*}
		 J^{n,N}_1&=&\frac{C}{N^{q/2}}\mathbb{E}\int_{0}^{T}\left[\int_{s/N}^{s}|f(r,X^n_r,Y^n_r,Z^n_r)|^2 dr \right]^{q/2}(dk^n_s+dk_s)\\
		&\le & \frac{C}{N^{q/2}} \mathbb{E}\left((k^n_T+k_T)\left[\int_{0}^{T}|f(r,X^n_r,Y^n_r,Z^n_r)|^2 dr \right]^{q/2} \right) \\
		&\le & \frac{C}{N^{q/2}}\left( \mathbb{E}(k^n_T+k_T)^{\frac{2}{2-q}}\right)^{\frac{2-q}{2}} \left( \mathbb{E}\int_{0}^{T}|f(r,X^n_r,Y^n_r,Z^n_r)|^2 dr  \right)^{q/2}.
	\end{eqnarray*}
	 By the linear growth of $f$ in its third variable and Lipschitz continuity with respect to the fourth argument, we get
	\begin{eqnarray*}
	J^{n,N}_1&\le& \frac{C}{N^{q/2}}\left( \mathbb{E}(k^n_T+k_T)^{\frac{2}{2-q}}\right)^{\frac{2-q}{2}} \left(1+\mathbb{E}\sup_{r\in [0,T]}|Y^n_r|^2 + \mathbb{E}\int_{0}^{T}\|Z^n_r\|^2 dr \right)^{q/2}\\
		&\le& \frac{C}{N^{q/2}}
	\end{eqnarray*}
	where in the last line we have used inequalities (\ref{boundXn}), (\ref{bound X and k}) and (\ref{bound Y^n }).
	
	Concerning $J^{n,N}_2$, we use H\"older's inequality, the monotony of $k^n$ and the linear growth condition on $h$, to obtain
	\begin{eqnarray*}
		&& J^{n,N}_2= C\,\mathbb{E}\int_{0}^{T}(k^n_s-k^n_{s/N})^{q/2}\left(\int_{s/N}^{s}|h(r,X^n_r,Y^n_r)|^2dk^n_r \right)^{q/2}(dk^n_s+dk_s)\\
		&\le& C\,\mathbb{E}\left(\int_{0}^{T}|h(r,X^n_r,Y^n_r)|^2dk^n_r \right)^{q/2}\sum_{i=1}^{N}\int_{r_{i-1}}^{r_i}(k^n_s-k^n_{s/N})^{q/2}(dk^n_s+dk_s)\\
		&\le& C\left(\mathbb{E}\int_{0}^{T}|h(r,X^n_r,Y^n_r)|^2dk^n_r \right)^{q/2}\left[ \mathbb{E}\left( \sum_{i=1}^{N}\int_{r_{i-1}}^{r_i}(k^n_s-k^n_{s/N})^{q/2}(dk^n_s+dk_s)\right)^{\frac{2}{2-q}}\right]^{\frac{2-q}{2}} \\
		&\le & C\left(\mathbb{E}\int_{0}^{T}|h(r,X^n_r,Y^n_r)|^2dk^n_r \right)^{q/2}\left[ \mathbb{E}\left( \sum_{i=1}^{N}(k^n_{r_i}-k^n_{r_{i-1}})^{q/2}(k^n_{r_i}+k_{r_i}-k^n_{r_{i-1}}-k_{r_{i-1}})\right)^{\frac{2}{2-q}}\right]^{\frac{2-q}{2}} \\
		&\le & C\left(\mathbb{E}\int_{0}^{T}1+|Y^n_r|^2dk^n_r \right)^{q/2}\left[ \mathbb{E}\left( \sum_{i=1}^{N}(k^n_{r_i}-k^n_{r_{i-1}})^{q/2}(k^n_{r_i}+k_{r_i}-k^n_{r_{i-1}}-k_{r_{i-1}})\right)^{\frac{2}{2-q}}\right]^{\frac{2-q}{2}} \\
		&\le & C\left(\mathbb{E}k^n_T+\mathbb{E}\int_{0}^{T}|Y^n_r|^2dk^n_r \right)^{q/2}\left[ \mathbb{E}\left( \sum_{i=1}^{N}(k^n_{r_i}-k^n_{r_{i-1}})^{q/2}(k^n_{r_i}+k_{r_i}-k^n_{r_{i-1}}-k_{r_{i-1}})\right)^{\frac{2}{2-q}}\right]^{\frac{2-q}{2}}.
	\end{eqnarray*}
	Hence, by inequalities (\ref{boundXn}) and (\ref{bound Y^n }), we see that
	\begin{eqnarray*}
		J^{n,N}_2&\le & C\left[ \mathbb{E}\left( \sum_{i=1}^{N}(k^n_{r_i}-k^n_{r_{i-1}})^{q/2}(k^n_{r_i}+k_{r_i}-k^n_{r_{i-1}}-k_{r_{i-1}})\right)^{\frac{2}{2-q}}\right]^{\frac{2-q}{2}}.
	\end{eqnarray*}
	Taking into account the convergence of $k^n$ to $k$ (Remark \ref{rq cv} (i)), estimates (\ref{boundXn}) and (\ref{bound X and k}), we pass to the limit as $n$ goes $+\infty$ then we use the Lebesgue dominated convergence theorem to get
	\begin{align*}
		\limsup_{n\to \infty}J^{n,N}_2&\le  C\left[ \mathbb{E}\left( \sum_{i=1}^{N}(k_{r_i}-k_{r_{i-1}})^{q/2}(k_{r_i}-k_{r_{i-1}})\right)^{\frac{2}{2-q}}\right]^{\frac{2-q}{2}}
  \\
		&\le  C\left[ \mathbb{E}\left( \max_{i=1,\hdots, N}(k_{r_i}-k_{r_{i-1}})^{q/2}k_T\right)^{\frac{2}{2-q}}\right]^{\frac{2-q}{2}}
 \\
		&\le   C\left[ \mathbb{E} \max_{i=1,\hdots, N}(k_{r_i}-k_{r_{i-1}})^{\frac{q}{2-q}}k_T^{\frac{2}{2-q}}\right]^{\frac{2-q}{2}}
 \\
		&\le  C\left[ \mathbb{E} \max_{i=1,\hdots, N}(k_{r_i}-k_{r_{i-1}})^{\frac{2q}{2-q}}\right]^{\frac{2-q}{4}}\left[ \mathbb{E}k_T^{\frac{4}{2-q}}\right]^{\frac{2-q}{4}}.
	\end{align*}
	For $J^{n,N}_3$, we have
	\begin{eqnarray*}
		J^{n,N}_3&=& C\,\mathbb{E}\int_{0}^{T}\left|\int_{s/N}^{s}Z^n_rdW_r\right|^q (dk^n_s+dk_s)\\
		&=& C\,\mathbb{E}\sum_{i=1}^{N}\int_{r_{i-1}}^{r_i}\left|\int_{s/N}^{s}Z^n_rdW_r\right|^q (dk^n_s+dk_s)\\
		&\le & C\,\sum_{i=1}^{N}\mathbb{E}\sup_{s\in[r_{i-1},r_i]}\left|\int_{s/N}^{s}Z^n_rdW_r\right|^q(k^n_{r_i}-k^n_{r_{i-1}}+k_{r_i}-k_{r_{i-1}})\\
		&\le & C \sum_{i=1}^{N}\left(\mathbb{E}\sup_{s\in[r_{i-1},r_i]}\left|\int_{r_{i-1}}^{s}Z^n_rdW_r\right|^2 \right)^{\frac{q}{2}}\left( \mathbb{E}(k^n_{r_i}-k^n_{r_{i-1}}+k_{r_i}-k_{r_{i-1}})^{\frac{2}{2-q}} \right)^{\frac{2-q}{2}}.\\
	\end{eqnarray*}
	Using the Burkholder-Davis-Gundy inequality, we obtain
	\begin{eqnarray*}
		J^{n,N}_3&\le & C \sum_{i=1}^{N}\left(\mathbb{E}\int_{r_{i-1}}^{r_i}\left\|Z^n_r\right\|^2dr \right)^{\frac{q}{2}}\left( \mathbb{E}(k^n_{r_i}-k^n_{r_{i-1}}+k_{r_i}-k_{r_{i-1}})^{\frac{2}{2-q}} \right)^{\frac{2-q}{2}}.
	\end{eqnarray*}
	Again by H\"older's inequality, we find
	\begin{eqnarray*}
		J^{n,N}_3&\le & C \left(\sum_{i=1}^{N}\mathbb{E}\int_{r_{i-1}}^{r_i}\left\|Z^n_r\right\|^2dr \right)^{\frac{q}{2}}\left(\sum_{i=1}^{N} \mathbb{E}(k^n_{r_i}-k^n_{r_{i-1}}+k_{r_i}-k_{r_{i-1}})^{\frac{2}{2-q}} \right)^{\frac{2-q}{2}}\\
		&\le & C\left(\mathbb{E}\int_{0}^{T}\left\|Z^n_r\right\|^2dr \right)^{\frac{q}{2}}\left(\sum_{i=1}^{N} \mathbb{E}(k^n_{r_i}-k^n_{r_{i-1}}+k_{r_i}-k_{r_{i-1}})^{\frac{2}{2-q}} \right)^{\frac{2-q}{2}}.
	\end{eqnarray*}
	Keeping in mind inequality (\ref{bound Y^n }) and the convergence of $k^n$ to $k$, we pass to the limit as $n\to+\infty$, to obtain
	\begin{eqnarray*}
		\lim_{n\to \infty}J^{n,N}_3&\le& C \left(\sum_{i=1}^{N} \mathbb{E}(k_{r_i}-k_{r_{i-1}})^{\frac{2}{2-q}} \right)^{\frac{2-q}{2}}\\
		&\le & C \left( \mathbb{E}\max_{i=1,N}(k_{r_i}-k_{r_{i-1}})^{\frac{q}{2-q}}\sum_{i=1}^{N}(k_{r_i}-k_{r_{i-1}}) \right)^{\frac{2-q}{2}}\\
		&\le & C \left( \mathbb{E}\max_{i=1,N}(k_{r_i}-k_{r_{i-1}})^{\frac{q}{2-q}} k_T \right)^{\frac{2-q}{2}}\\
		&\le & C \left( \mathbb{E}k_T^2\right)^{\frac{2-q}{4}} \left( \mathbb{E}\max_{i=1,N}(k_{r_i}-k_{r_{i-1}})^{\frac{2q}{2-q}}  \right)^{\frac{2-q}{4}}\\
		&\le & C \left( \mathbb{E}\max_{i=1,N}(k_{r_i}-k_{r_{i-1}})^{\frac{2q}{2-q}}  \right)^{\frac{2-q}{4}}.
	\end{eqnarray*}
	This completes the proof of Lemma \ref{lem3}.	
\end{proof}
\begin{lemma}\label{lem4} Let $R^n$ be the process defined by (\ref{R^n}).  Under assumptions {\rm{(A.1)}} and {\rm{(A.4),}}  the following inequality holds
	$$\limsup_{n\to \infty}\mathbb{E}\sup_{s\in[0,T]}\int_{s}^{T}e^{2\lambda r}dR^n_r\le 0.$$
\end{lemma}
\begin{proof}{}
	Set $h_r=h(r,X_r,Y_r)$ and $|h|_\infty=\sup_{r\in [0,T]}|h_r|$, we have
	\begin{eqnarray*}
		<Y^n_r-Y_r,h_r>&=& <Y^{n,N}_r-Y^{N}_r,h_r-h_r^N>+<Y^N_r-Y_r,h_r>\\
		&&+<Y^{n,N}_r-Y^{N}_r,h_r^N>+<Y^{n}_r-Y^{n,N}_r,h_r> .
	\end{eqnarray*}
	\begin{eqnarray*}
		\mathbb{E}\left(\sup_{s\in [0,T]}\int_{s}^{T} e^{2\lambda r}dR^n_r\right)&=& \mathbb{E}\left(\sup_{s\in [0,T]}\int_{s}^{T} e^{2\lambda r}<Y^n_r-Y_r,h(r,X_r,Y_r)(dk^n_r-dk_r)>\right) \\
		&\le&\mathbb{E}\left[\left((|Y^n|_\infty+|Y|_\infty)|h-h^N|_\infty +|Y^N-Y|_\infty|h|_\infty\right)e^{2\lambda T}(k^n_T+k_T)  \right] \\
		 &&+\mathbb{E}\left(\sup_{s\in[0,T]}\sum_{i=1}^{N}<Y^n_{r_{i-1}}-Y_{r_{i-1}},h_{r_{i-1}}>\int_{s\wedge r_{i-1}}^{r_i}e^{2\lambda r} d(k^n_r-k_r) \right)\\
		&&+\mathbb{E}\left(e^{2\lambda T}|h|_\infty\int_{0}^{T}|Y^{n,N}_r-Y^{n}_r|(dk^n_r+dk_r) \right).
	\end{eqnarray*}
	Let $1<q<2$. Using H\"older's inequality repeatedly to obtain
	\begin{eqnarray}\label{estimate int dR^n}
		&&\mathbb{E}\left(\sup_{s\in [0,T]}\int_{s}^{T} e^{2\lambda r}dR^n_r\right)\le \left[ \mathbb{E}\left[e^{2\lambda T}(|Y^n|_\infty+|Y|_\infty)\right]^2\right]^{\frac{1}{2}}\left[\mathbb{E}(k^n_T+k_T)^4 \right]^{\frac{1}{4}}  \left[ \mathbb{E}|h-h^N|_\infty^4\right]^{\frac{1}{4}}\nonumber\\
		&&  +\left[ \mathbb{E}\left[e^{2\lambda T}(k^n_T+k_T)|h|_\infty \right]^{2}\right]^{\frac{1}{2}}  \left[ \mathbb{E}|Y^N-Y|_\infty^2\right]^{\frac{1}{2}}\nonumber\\
		&&+2Ne^{2\lambda T}(1+\lambda T)\left[ \mathbb{E}\left[e^{2\lambda T}(|Y^n|_\infty+|Y|_\infty) \right]^{2}\right]^{\frac{1}{2}}\left[ \mathbb{E}|h|_\infty^4\right]^{\frac{1}{4}}  \left(\mathbb{E}\left[\sup_{s\in [0,T]}|k^n_{s}-k_{s}|^4 \right]  \right)^{\frac{1}{4}}\nonumber\\
		&& +e^{2\lambda T}\left[ \mathbb{E}|h|_\infty^{\frac{2q}{q-1}}\right]^{\frac{q-1}{2q}}\left[\mathbb{E}\left(k^n_T+k_T \right)^2  \right]^{\frac{q-1}{2q}}  \left[ \mathbb{E}\int_{0}^{T}|Y^{n,N}_r-Y^{n}_r|^q(dk^n_r+dk_r)\right]^{\frac{1}{q}}.
	\end{eqnarray}
	 The linear growth hypothesis on $ h$ combined with estimate (\ref{bound Y}) show that for every $ p\ge 1$,
	 $$\mathbb{E}|h|_\infty^p=\mathbb{E}\sup_{r\in[0,T]}|h(r,X_r,Y_r)|^p\le C(1+\mathbb{E}\sup_{r\in[0,T]}|Y_r|^p)<+\infty.$$
	On the other hand, we use inequalities (\ref{boundXn}), (\ref{bound X and k}), (\ref{bound Y^n }) and (\ref{bound Y}) along with Lemma \ref{lem3} then we pass to the limit as $n$ goes to $+\infty$ in  inequality (\ref{estimate int dR^n}) to get, for all $N\in \mathbb{N}^*$,
	\begin{eqnarray*}
		\limsup_{n\to +\infty}\mathbb{E}\left(\sup_{s\in [0,T]}\int_{s}^{T} e^{2\lambda r}dR^n_r\right)&\le& C\left(\mathbb{E}|h-h^N|_\infty^4 \right)^{1/4}+C\left(\mathbb{E}|Y^N-Y|_\infty^2 \right)^{1/2}\\
		&& +\left[ \frac{C}{N^{q/2}}+ C\left[\mathbb{E}\max_{i=1,N}\left(k_{r_i}-k_{r_{i-1}} \right)^{\frac{2q}{2-q}}  \right]^{\frac{2-q}{4}}\right]^{1/q} .
	\end{eqnarray*}
	Since the integrands are uniformly integrable, then passing to the limit as $N\to +\infty$, we get the result. Lemma \ref{lem4} is proved. \end{proof}

Now, combining inequality (\ref{main ineq}) with Lemmas \ref{lem1}, \ref{lem2} and  \ref{lem4}, we complete the proof of Theorem \ref{cvgcebsde}.
\end{proof}

\subsection{ Convergence of the penalized PDE }
This subsection is devoted to an application of our convergence of the BSDE. Namely, we will establish the convergence of a viscosity solution of the following systems
\begin{equation}\label{pdeseq}
\left\{
\begin{aligned}
\displaystyle
&\frac{\partial u^{n}_{i}}{\partial t}(t,x) +\mathcal{L}\, u^{n}_{i}(t,x)+
f_{i}(t,x,u^{n}(t,x),(\nabla u^n\sigma)(t,x))\\
&\qquad\qquad -n\,<\delta(x),\nabla
u^{n}_{i}(t,x)+\nabla l(x)>
h_{i}(t, x, u^n(t,x))  = 0\,, \\
&u^{n}(T, x)  =  g(x)\,, 1\leq i\leq m, 0\leq t\leq T, x\in\R^{d},
n\in\N \hfill\cr
\end{aligned}\right.
\end{equation}
to a viscosity solution of a system of the form
 \begin{equation}\label{pdelimit}
 \left\{
 \begin{matrix}
 \displaystyle\frac{\partial u_{i}}{\partial t}(t,x) +\mathcal{L}u_{i}(t,x)+
 f_{i}(t,x,u(t,x),(\nabla u\sigma)(t,x))  = 0\,,\  1\leq i\leq m\,,\,\,(t,x)\in[0,T)\times D\,,
 \hfill\\
 \noalign{\medskip}
 u(T,x)  =  g(x)\,,\ \  x\in D\,,\hfill\\
 \displaystyle\frac{\partial u}{\partial n}(t,x)+h(t, x, u(t,x))=0
 \,,\,\,\forall
 (t,x)\in[0,T)\times\partial D\, .\hfill
 \end{matrix}\right. 
 \end{equation}
  Since, we consider viscosity solutions, we introduce the following condition
  \vskip 0.2cm \noindent
  (A.5) \ $f_i$, the $i$-th coordinate of $f$, depends only on the $i$-th row of the matrix $z$.
 \vskip 0.2cm
For the self-contained, we recall the definition of the viscosity solution  of system (\ref{pdelimit}).
 \begin{definition}
 	\begin{itemize}
 		\item[(i)] $u\in \mathcal{C}([0,T]\times \bar{D},\mathbb{R}^m)$ is called a viscosity sub solution of system $(\ref{pdelimit})$ if $u_i(T,x)\le g_i(x)$, $x\in\bar{D}$, $1\le i\le m$ and,   for any $1\le i\le m$, $\varphi \in \mathcal{C}^{1,2}([0,T]\times\mathbb{R}^d)$, and $(t,x)\in (0,T]\times\bar{D}$ at which $u_i-\varphi$ has a local maximum, one has
 		$$- \frac{\partial \varphi}{\partial t}(t,x)-\mathcal{L}\varphi(t,x)-
 		f_{i}(t,x,u(t,x),(\nabla\varphi\sigma)(t,x)) \le 0,\quad \text{if}\quad x\,\in\, D, $$
 		\begin{eqnarray*}		
 			\begin{split}
 				&\min\left( -\frac{\partial \varphi}{\partial t}(t,x)-\mathcal{L}\varphi(t,x)-
 				f_{i}(t,x,u(t,x),(\nabla\varphi\sigma)(t,x)), \right. \\
 				& \left. -\frac{\partial \varphi}{\partial n}(t,x)-h_i(t, x, u(t,x))  \right)\le 0,\quad \mbox{if}\   x\in \partial D.
 			\end{split}
 		\end{eqnarray*}
 		\item[(ii)] $u\in \mathcal{C}([0,T]\times \bar{D},\mathbb{R}^d)$ is called a viscosity super-solution of (\ref{pdelimit}) if $u_i(T,x)\ge g_i(x)$, $x\in\bar{D}$, $1\le i\le m$ and, for any $1\le i\le m$, $\varphi \in \mathcal{C}^{1,2}([0,T]\times \mathbb{R}^d)$, and $(t,x)\in (0,T]\times\bar{D}$ at which $u_i-\varphi$ has a local minimum, one has
 		$$- \frac{\partial \varphi}{\partial t}(t,x) -\mathcal{L}\varphi(t,x)-
 		f_{i}(t,x,u(t,x),(\nabla\varphi\sigma)(t,x)) \ge 0,\quad \text{if}\quad x\,\in\, D, $$
 		\begin{eqnarray*}		
 			\begin{split}
 				& \max\left(- \frac{\partial \varphi}{\partial t}(t,x) -\mathcal{L}\varphi(t,x)-
 				f_{i}(t,x,u(t,x),(\nabla\varphi\sigma)(t,x)), \right. \\
 				& \left. -\frac{\partial \varphi}{\partial n}(t,x)-h_i(t, x, u(t,x))   \right)\ge 0,\quad \mbox{if}\  x\in \partial D.
 			\end{split}
 		\end{eqnarray*}
 		\item[(iii)]  $u\in \mathcal{C}([0,T]\times \bar{D},\mathbb{R}^m)$ is called a viscosity solution of system (\ref{pdelimit}) if it is both a viscosity sub and super-solution.	
 	\end{itemize}
 \end{definition}

We recall a continuity of the map $(t,x) \mapsto Y_t^{t,x}$ where $Y^{t,x}$ is  the solution of BSDE (\ref{baklimit}). This continuity has been proved in \cite{PardRas}.
\begin{proposition}\label{Pena: conti Y^{t,x}}
	(\cite{PardRas})
Under assumptions \rm{(A.1)} and \rm{(A.4)}, the mapping $(t,x)\to Y^{t,x}_t$ is continuous.
\end{proposition}

 Our second main result is

\begin{theorem}\label{main result}
	Assume \rm{(A.1), (A.4)} and \rm{(A.5)}. There exist a sequence of continuous functions $u^n:[0,T]\times\mathbb{R}^d\to \mathbb{R}^m$ and a function $u:[0,T]\times\bar{D}\to \mathbb{R}^m$ such that: $u^n$ is a viscosity solution to system $(\ref{pdeseq})$, $u$ is a viscosity solution to system $(\ref{pdelimit})$ and the following convergence holds for every $(t,x)\in [0,T]\times\bar{D}$
	$$\lim_{n\to +\infty}u^n(t,x)=u(t,x).$$	
\end{theorem}

\begin{proof}{}
We set, \begin{equation}
u^n(t,x):=Y^{t,x,n}_t\quad \mbox{and}\quad u(t,x):=Y^{t,x}_t.
\end{equation}
It follows from Theorem 3.2 of \cite{Pard} that $u^n$ is a viscosity solution of PDEs (\ref{pdeseq}). Thanks to \cite{PardRas,PardZhang}, $u$ is a viscosity solution of PDEs (\ref{pdelimit}). Further, we have for each $(t,x)\in [0,T]\times \bar{D}$
\begin{eqnarray*}
|u^n(t,x)-u(t,x)|^2=|Y^{t,x,n}_t-Y^{t,x}_t|^2\le \mathbb{E}\sup_{s\in [0,T]}|Y^{t,x,n}_s-Y^{t,x}_s|^2.
\end{eqnarray*}
Thanks to Theorem \ref{cvgcebsde}, we have
$lim_{n\rightarrow\infty}\mathbb{E}\sup_{s\in [0,T]}|Y^{t,x,n}_s-Y^{t,x}_s|^2 = 0$.
It follows that,
$$\lim_{n\to+\infty}u^n(t,x)=u(t,x).$$
The theorem is proved.
\end{proof}
\begin{remark}
 When $u$ is the unique viscosity solution of system $(\ref{pdelimit})$, then it is constructible by penalization. This is the case when  $d=d'$, since by \cite[Theorem 5.43, p 423]{PardRas2014} or \cite[Theorem 5.1]{PardRas} the viscosity solution of system (\ref{pdelimit}) is unique.
\end{remark}

\appendix

\end{document}